\documentclass[12pt,a4paper,twoside]{amsart}
\usepackage{a4wide,times, amsmath,amsbsy,amsfonts,amssymb,
stmaryrd,amsthm,mathrsfs,graphicx, amscd, tikz-cd, cancel}
\usepackage{dsfont}
\usepackage[all]{xy}
\usepackage{xypic}
\usepackage{yhmath}
\usepackage{mathrsfs}
\usepackage{subfig}
\usepackage{floatrow}
\usepackage{multirow}
\usepackage{CJKutf8}
\usepackage[linesnumbered,lined,boxed,commentsnumbered]{algorithm2e}
\usepackage[utf8]{inputenc}
\usepackage{fancyhdr}
\usepackage{lastpage}  
\usepackage{chemfig}

\usepackage[normalem]{ulem}
\usetikzlibrary{matrix,arrows,decorations.pathmorphing}
\usepackage[active]{srcltx}
\usepackage[
	hypertexnames=false,
	hyperindex,
	pagebackref,
	pdftex,
	breaklinks=true,
	bookmarks=true,
	colorlinks,
	linkcolor=blue,
	citecolor=red,
	urlcolor=red,
]{hyperref}
\usepackage[mathscr]{eucal}

\newcommand\SO{\operatorname{SO}}

\newcommand\Img{\operatorname{Im}}
\newcommand\Rea{\operatorname{Re}}

\newcommand\cro{\operatorname{cr}}

\newcommand\VR{\operatorname{VR}}
\newcommand\Ctxt{C}
\newcommand\Ntxt{N}

\newcommand\Cds{\mathds{C}}

\newcommand\Kds{\mathds{K}}

\newcommand\Hds{\mathds{H}}

\newcommand\Rds{\mathds{R}}

\newcommand\Zds{\mathds{Z}}

\newcommand\varparallel{\mathrel{/\mskip-2.5mu/}}

\makeatletter

\newcommand{\Rnum}[1]{\expandafter\@slowromancap\romannumeral #1@}

\newtheorem{theorem}{Theorem}[section]
\newtheorem{Def}[theorem]{Definition}  
\newtheorem{lem}[theorem]{Lemma}  
\newtheorem{thm}[theorem]{Theorem}  
\newtheorem{cor}[theorem]{Corollary}  
\newtheorem{prop}[theorem]{Proposition}

\theoremstyle{definition}
\newtheorem{rmk}[theorem]{Remark}

\begin{document}
\title{New Aspects of Analyzing Amyloid Fibrils}  
\author{Xiaoxi Lin, Yunpeng Zi, Fengling Li and Jingyan Li}  
\date{}  
\maketitle

\begin{abstract}
	This is a summary of mathematical tools we used in research of analyzing the structure of proteins with amyloid form \cite{xi2024Top}.
	We defined several geometry indicators on the discrete curve namely the hop distance, the discrete curvature and the discrete torsion. Then, we used these indicators to analyze the structure of amyloid fibrils by regarding its peptide chains as discrete curves in $\Rds^3$. We gave examples to show that these indicators give novel insights in the characterization analysis of the structure of amyloid fibrils, for example the discrete torsion can detect the hydrogen bonds interactions between layers of amyloid fibril. {Moreover,} the topological tool performs better than the root mean square deviation (RMSD) in quantifying the difference of the structure of amyloid fibrils, etc.
\end{abstract}

\section{Introduction}

This is a summary of mathematical tools we used in research of analyzing the structure of proteins with amyloid form \cite{xi2024Top}. Amyloids are aggregates of proteins with a fibrillar morphology, which results from the protein misfolding. This misfolding may cause the protein lose its original structure and physiological functions, and get new structure and physiological functions. Amyloids have been implicated in both physiological processes and pathological diseases, which can be termed as pathological amyloids and functional amyloids \cite{hammer2008amyloids}, respectively. The abnormal depositions of pathological amyloids in various tissues are associated with kinds of diseases known as amyloidosis. To date, there are more than 30 human proteins have been found to form amyloidosis \cite{chiti2017protein}, such as transthyretin (TTR), $\alpha$-synuclein \cite{irvine2008protein} and $\beta$ amyloid peptide \cite{hamley2012amyloid}.

Studies over the past 2 decades have made a great contribution to explore the molecular and physiochemical mechanisms of amyloidogenesis and its role in biology, including the overall architecture of amyloid fibril \cite{eisenberg2012amyloid, chen2017amyloid}, the interactions related to the stability of fibril structures \cite{chen2017amyloid}, the process of amyloid formation \cite{michaels2018chemical, xue2008systematic}, the amyloid core and flanking regions of fibril structures \cite{morgan2022transient, tompa2009structural}, the conformational dynamics \cite{matthes2023molecular, michaels2020dynamics}, etc. However, challenges faced by studies on amyloids are still multiple and significant \cite{strodel2021amyloid}. One of them is that how the structure of amyloid fibril determines its new biological functions. It is inevitable to analyze the structure of amyloid fibrils, and then to explore the relationship between the structure and its corresponding functions. The rapid development of structural biology methods, including cryo-electron microscopy (cryo-EM) and solid-state nuclear magnetic resonance (ssNMR), makes it possible to analyze the structure of amyloid fibrils at atomic level \cite{sawaya2021expanding}.

In this paper, we gave several ways of analyzing the structure of amyloid fibrils using the tools from topology and geometry, which could give insights into further studies on amyloid fibrils.
Examples in this paper are all related to the structure of amyloid fibrils formed from TTR (ATTR).
TTR is a homotetramer with a dimer of dimers and high $\beta$-sheet content. The native TTR transports the thyroxin hormone thyroxine and retinol to the liver. The misfolding and aggregation of TTR result in the formation of ATTR and amyloidosis \cite{zeldenrust2010familial}. The extracellular despositions of misfolded wild type TTR could lead to senile systemic amyloidosis, while the extracellular despositions of misfolded mutational variants, such as V30M and V122I, are associated with hereditary ATTR amyloidosis, familial amyloid polyneuropathy and familial amyloid cardiomyopathy. Before introducing our study, let us give some biological and chemical knowledges about the structure of protein and a {historical} introduction about the {intersection of geometry, topology and biology}.

\subsection{Primary Structures of Proteins}\label{subsec:strprotein}
Protein is a class of large biomolecules, which plays important role in biochemical reactions. Amino acids (AA) are the constituent units of proteins.
Specifically, each AA consists of an amino group (-$\text{NH}_2$), a carboxyl group (-COOH) and a side chain group (-R). These three parts are connected through a carbon atom, which we call it the $\alpha$-Carbon and denote it as $\text{C}_\alpha$. The type of AA is entirely determined by its side chain group. AAs form peptide bonds through dehydraton condensation, and then determine the primary structure of proteins, called the peptide chain which is shown in Fig.\ref{fig: peptide chain}. An individual AA in the peptide chain is called a residue. Each peptide chain has its backbone, which consists of amino nitrogen, $\text{C}_\alpha$, carboxyl carbon and carboxyl oxygen of each residue. Note that a peptide chain is naturally oriented by its formation, that is, a peptide chain always begin with a amino nitrogen and end with a carbonyl carbon.


\begin{figure}[ht]
	\centering
	\includegraphics[height=4cm, width=10.5cm]{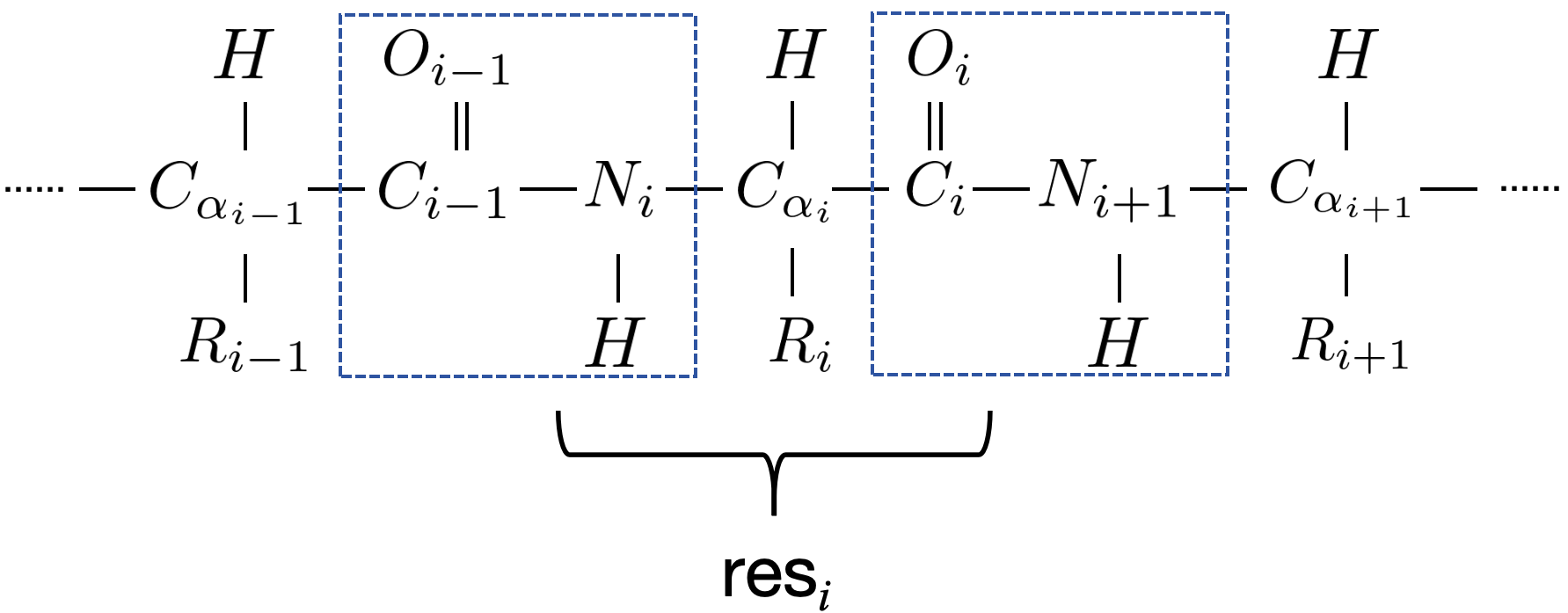}
	\caption{A peptide chain. The backbone of this peptide chain is the set $\{\dots, \text{C}_{\alpha_{i-1}}, \text{C}_{i-1}, \text{O}_{i-1},  \text{N}_{i}, \text{C}_{\alpha_{i}}, \text{C}_{i}, \text{O}_{i}, \dots\}.$ Blue rectangles display the peptide bonds. Here, $\text{res}_{i}$ represents the $i$-th residue of the peptide should be given in the caption.}
	\label{fig: peptide chain}
\end{figure}

\subsection{Geometry in Biology}\label{subsec:Geo-Bio}
The geometry plays an important role in the study of the structure of proteins and even other biomolecules. Scientists found that three-dimensional atomic structures of these molecules provide rich information for understanding how these molecules carry out their biological functions. Among all these research, we found several cutting points that geometry can enter the study of the structures of proteins or biomolecules.

As commonly understood, a molecule consists of moving particles(atomic nuclei and electrons) held together by electrostatic and magnetic forces in nature, lacking definitive boundaries. However, visualizing a molecule is still a very useful chemical concept and has proved its value in medicinal chemistry.

The van der Waals radius of an atom, denoted as $r_w$, is the radius of an imaginary hard sphere representing the distance of closest approach for another atom. The van der Waals surface of a molecule is an abstract representation or model of that molecule by treating each atom as a sphere with the van der Waals radius. This model is named for Johannes Diderik van der Waals, a Dutch theoretical physicist and thermodynamicist. Several models have made improvements towards the van der Waals surfaces. The CPK models is developed by and named for Robert Corey, Linus Pauling, and Walter Koltun in the paper \cite{corey1953molecular}. In this model, atoms were represented by spheres whose radii were proportional to the van der Waals radii of the atoms, and whose center-to-center distances were proportional to the distances between the atomic nuclei in the same scale. The \textit{Lee-Richards surface model}, raised by B.Lee and F.M.Richards in the paper \cite{lee1971interpretation}, was given by the center of a radius $r_w$ sphere when this ball was rolling everywhere along the van der Waals surface of the molecule. While Lee-Richards surface might not be smooth, the \textit{Connolly's surface model}, introduced by M.L.Connolly in \cite{connolly1983analytical}, was constructed by replacing the center by the `front' of the same sphere.

On the other hand, it was found in biology that low-energy elements occur more frequently than others in the 3D structure of proteins. This interesting phenomenon was summarized as the \textit{Boltzmann like statistics for proteins}, i.e.,
\begin{equation}\label{eq:Boltzmann-statistics}
	\hbox{OBSERVED OCCURRENCE} \sim  \text{exp}(-\frac{\hbox{ENERGY}}{RT^{\ast}})
\end{equation}
where $T^{\ast}$ is the `conformational temperature' of protein statistics which is close to room temperature, and $R$ is the Boltzmann constant \cite{finkelstein1995protein,finkelstein1995boltzmann}. From this point of view, geometry objects could enter as representations of biomolecules. For example, the most commonly known theory in this side is the Ramachandran plot for proteins raised in \cite{RAMACHANDRAN196395}. In this theory, people focused on two backbone dihedral angles $\psi$ and $\phi$ of residues in proteins. The Ramachandran plot visualized energetically allowed regions for $\psi$ against $\phi$.

The existence of hydrogen bonds, especially the backbone hydrogen bond (BHB), is also important in the formation of the structure of proteins and the realization of the functions of proteins. Following the same philosophy of Equation \ref{eq:Boltzmann-statistics}, the \textit{Backbone Free Energy Estimator} was introduced by R.Penner in \cite{penner2020backbone,penner2022protein}. In this theory, he constructed the unit 3D frames on both sides of BHB with $x-y$ plane is the peptide plane and $z$ axis is the associated cross product. Hence, he could represent a BHB as an element on the Lie group $\SO_3(\Rds)$. The distribution of BHBs on $\SO_3(\Rds)$ could be computed from a well-selected data set.

The energy also appears in the studies and simulations of variations of protein conformations. This is a joint field of differential geometry, the shape analysis and biology. The shape space is somehow like the moduli space in mathematics. Since the energy usually appears as a sum of the potential of interactions that only depend on the Euclidean distance between pairs of atoms in the protein. Two conformations can only be distinguished up to rigid-body transformations. This observation gives the idea of the shape space.
\textit{The shape space} is a quotient manifold by identifying each protein conformation with the orbit space of configuration space of all protein conformations by its rigid-body transformations.
The dynamics of the protein conformation could be treated as the flow or the deformation on the shape space. There are several ways of constructing the shape spaces, for example using the backbone dihedral angles, using the point cloud and the Normal mode analysis (NMA) \cite{brooks1983harmonic,go1983dynamics,brooks1985normal}. The {Backbone Free Energy Estimator} \cite{penner2020backbone,penner2022protein} introduced above with the shape space $\SO_3(\Rds)$ is another example. Shape space is an important object in shape analysis and computer vision, we refer \cite{charon2022shape,laga2018survey,younes2010shapes} for more detailed introduction.

Finally, let us mention a recent work \cite{diepeveen2023riemannian} trying to bring the Riemannian geometry into the study of the protein conformations. In this paper, they constructed a smooth manifold from the quotient of the configuration space of non-planer atom coordinates by the group of affine translations. The different conformations of the proteins were represented as points on this manifold. An energy motivated and efficiently evaluated metric could be defined on this manifold such that making this manifold a Riemannian manifold. Then the dynamics of changing of protein conformations could be efficiently interpolated and extrapolated along non-linear geodesic curves.

\subsection{Discrete Curve Model}
Many topics in applied geometry like computer graphics and computer vision have close connection with discrete geometry, especially the discrete curves and discrete surfaces.
For example, the tasks like analysis of geometric data, its reconstruction, and further its manipulation and simulation.
Approximations of 3D-geometric notions play a crucial part in creating algorithms that can handle such tasks.
Similarly as the smooth realm where the curvature and torsion play important roles in parametrizing a smooth curve or surface, the estimation of curvatures and torsions of curves and surfaces is also important in applied geometry.

Recall that the curvature of a smooth curve $\gamma$ at a fixed point $p$ is defined as the inverse of the radius of the osculating circle of $\gamma$ at point $p$. It quantifies the extent of the bending of $\gamma$ at $p$. The torsion of the curve $\gamma$ is given by the 'percentage' between the normal bundle of the curve and the projection of the differential of the binormal bundle on the normal bundle. It measures the speed of a smooth curve moves away from it osculating plane. In the discrete case, we applied the ideas and methods from classical differential geometry instead of “simply” discretizing equations or using classical differential calculus. There are several versions of discretized curvature and torsion based on using different starting point of the classical differential geometry like from Frenet-frame \cite{carroll2014survey} and from geometric knot theory \cite{sullivan2008curves}.

In this paper, we mentioned two ways of applying properties of the discrete curve. The first, which we call it the truncated Hop distance between two discrete curves, is defined as a matrix composed of the differences of hop distances of vertices between two discrete curves. We will show below by example that treating a protein as a discrete curve and applying this simply defined and computed parameter give us an important property in the formation of amyloid fibril. We will show below by example that by treating a protein as a discrete curve, this simply defined and computed parameter give us an important property of the formation of amyloid fibril.

On the other hand, we introduced a version of discrete curvature and torsion from \cite{muller2021discrete} which is based on the osculating circle of a discrete curve. We will introduce an example that this version of torsion could detect the BHB between layers of amyloid fibril.


\subsection{Topological Model}


Topological Data Analysis (TDA) is a mathematical field focus on the `shapes' of the data. It provides a set of mathematical and computational techniques to analyze complex high-dimensional data, and identify key features of the data which may be difficult to be extracted with traditional methods. Recently, TDA has been witnessed applications in a variety of disciplines, such as computational chemistry \cite{townsend2020representation}, nano material \cite{xia2015persistent}, environmental science \cite{ofori2021application, ohanuba2023application}, natural language processing \cite{elyasi2019introduction}, image segmentation \cite{clough2020topological}, medical imaging \cite{singh2023topological, moraleda2019computational, pritchard2023persistent}, deep learning model \cite{rabbani2020topological, chen2022bscnets}, complex network \cite{gao2023hierarchical, wu2023metabolomic}, economy \cite{shultz2023applications}, biomolecular analysis \cite{dey2018protein, liu2022biomolecular}, the winning of D3R Grand Challenges \cite{nguyen2020review, nguyen2019mathematical} and the evolutionary mechanism of SARS-CoV-2 \cite{chen2020mutations}, etc. Compared with traditional dimensionality reduction methods, such as principal component analysis (PCA), TDA retains more effective information while reducing the dimensionality of data.

Up to now, various tools have been developed under the heading of TDA, such as (i) Mapper. Mapper first divides the original data into several clusters based on the preassigned filter functions. Then, two clusters are connected if they have common data points. The output of mapper is a combinatorial graph. Monica N. et al. analyzed the microarray gene expression data by adopting Mapper and disease-specific genomic analysis transformation, and successfully identified a unique subgroup of Estrogen Receptor-positive (ER+) breast cancers\cite{nicolau2011topology}. (ii) Persistent topological laplacians (PDL). PDL captures information about data based on the spectra of the laplacian operator. Specifically, the harmonic spectra relates to the Betti number of data, while the non-harmonic spectra provides extra geometry information \cite{gulen2023generalization, liu2023algebraic}. PDL-based approaches have been applied to analyze the structure of protein-protein interaction networks \cite{du2024multiscale}. (iii) Persistent homology (PH). PH is one of the most popularly used tool in TDA, which creates a multiscale family of simplicial complexes and describes the evolution of these complexes and their associated homology groups over scales. We will focus on the PH in the following materials.

Persistent homology has been widely applied in data science, such as crystalline material \cite{jiang2021topological}, chemical engineering \cite{smith2021topological}, especially computational biology \cite{liu2022biomolecular}. For example, the element-specific PH combined with deep learning winned the D3R Grand Challenges \cite{nguyen2020review, nguyen2019mathematical}, the element- and site-specific PH adopted for prediction the evolutionary mechanism of SARS-CoV-2 \cite{chen2020mutations} and hypergraph-based PH to predict protein-ligand binding affinity \cite{liu2021hypergraph}. Theories related to PH have been developed rapidly, including weighted PH \cite{ren2018weighted, bell2019weighted}, multiparameter PH \cite{guidolin2023morse, ripp2024multi, loiseaux2024framework}, etc.  Moreover, the open source packages, such as Gudhi \cite{maria2014gudhi}, Javaplex \cite{adams2014javaplex}, Ripser\cite{bauer2021ripser},  etc., have contributed to the wide range of applications in different areas of PH.
In this paper, we introduced basic definitions of the PH. More details can be found in \cite{zomorodian2004computing, edelsbrunner2008persistent}. Moreover, We will give an example of the application showing that these tools give better indicators to distinguish amyloid fibrils than the commonly used indicator RMSD.

\subsection{Arrangement of the Paper}
The reminder of this paper is organized as follows. In Section \ref{sec:hop-dis},  we defined the truncated Hop distance to compare the structure of two discrete curves. Moreover, we explained how to treat one layer of amyloid fibril as a discrete curve. And we suggest that the effect in the formation of amyloid fibril is relatively mild based on the result of the truncated Hop distance. Preliminaries of the definitions of curvature and torsion on the discrete curves are presented in Section \ref{sec:cur-tor}. By computing the discrete curvature and torsion on single layers of amyloid fibril, we suggest that the anomalies reflect slight affects of the layer-layer hydrogen bonds. Finally in Section \ref{sec:Top-Methods}, we introduced the distances between persistent diagrams (PDs), and showed that how these distances work in quantifying the difference of the structure of amyloid fibrils.
The atomic models for crystal X-ray diffraction structures of human TTR and cryo-EM structures of ATTR wt and ATTRv were obtained from Protein Data Bank (PDB) \cite{berman2000protein}.

\section{Discrete Curve and Hop Distance}\label{sec:hop-dis}

A \textit{discrete curve} is a polygonal curve in $\Rds^2$ or $\Rds^3$ given by its vertices via a map $\gamma:\Zds\to\Rds^3$. In this section, we will introduce a direct parameter that not only could distinguish two different discrete curves, but also showed meaning in analyzing the difference of the structure of protein with well-folded form and amyloid form.
\subsection{Hop distance}
Let $\gamma$ and $\gamma'$ be two different discrete curves, and $\gamma_i$ and $\gamma'_i$ be the corresponding points of vertices.
Let $d(\gamma_i, \gamma_{i+k})$ be the Euclidean distance between the $i$-th and $(i+k)$-th vertices, also called the $k$-hop distance between $\gamma_i$ and $\gamma_{i+k}$.
\begin{Def}
	The \textit{ $N$-truncated hop distance}, where $N$ is a fixed positive integer, between the two discrete curve $\gamma$ and $\gamma'$ is given by the $N\times N$ matrix $D$ with
	\begin{equation}
		D_{ij} = |d(\gamma_i, \gamma_j) - d(\gamma_i', \gamma_j')|.
	\end{equation}
\end{Def}


%

\subsection{Applications to Analysis of the Structure of Amyloid fibrils}
To explore structural or functional properties of proteins in practice, we mostly consider the heavy atoms of residues, the backbone of peptide chain, or the set of $C_\alpha\text{s}$ \cite{bramer2020atom, wang2020topology} to simplify the structure of protein. This simplication reduce the redundant and irrelevant information, but also preserve the significant structural features of proteins, which can improve the efficiency of research. In this case, we may treat a peptide chain as a discrete curve to explore its geometry invariants. It could be helpful to extract the structural features of protein,  and further, to analyze the functions of protein. In this example, we focused on the analysis of the structure of point mutant V30M TTR with well-folded form and amyloid form, as shown in Fig.\ref{fig: discrete curve}.

\begin{figure}[ht]
	\centering
	\includegraphics[height=6cm, width=10cm]{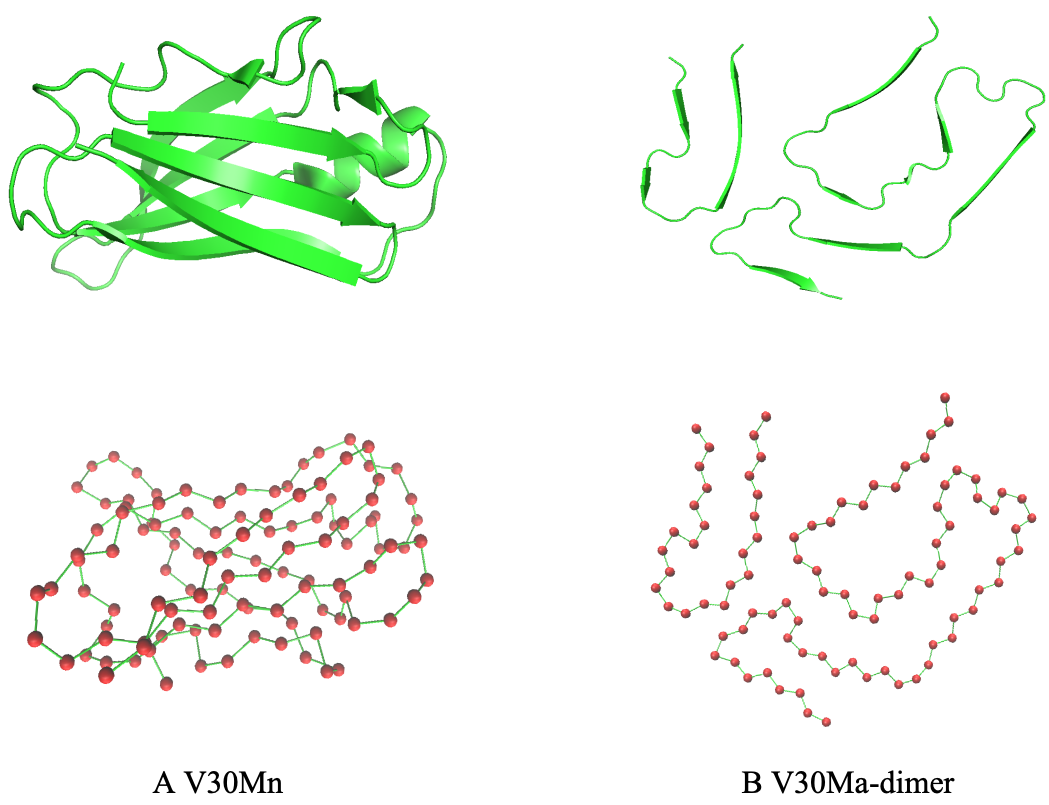}
	\caption{Convert peptide chains into discrete curves. (A) The top is the cartoon representation of chain A of TTR V30M mutant (PDB entry: V30Mn), and the bottom is its corresponding discrete curve. (B) The top is the cartoon representation of chain A of one amyloid form of V30Mn (V30Ma-dimer), and the bottom is its corresponding discrete curve. Red balls represent the $\text{C}_\alpha$ atoms. }
	\label{fig: discrete curve}
\end{figure}

It has been discovered that amyloid fibril serves as the key pathological entity in different neurodegenerative diseases, such as Alzheimer's disease and Parkinson's disease \cite{peng2020protein, jucker2013self}. However, the process by which amyloid fibril is formed is so far obscure, hindering further research into the causes of diseases. Here, we used the \textit{truncated hop distance} to explore the strength of effect in the formation of amyloid form, as shown in  Fig.\ref{fig: difference of hop distance}.

\begin{figure}[ht]
	\centering
	\includegraphics[height=5.5cm, width=6cm]{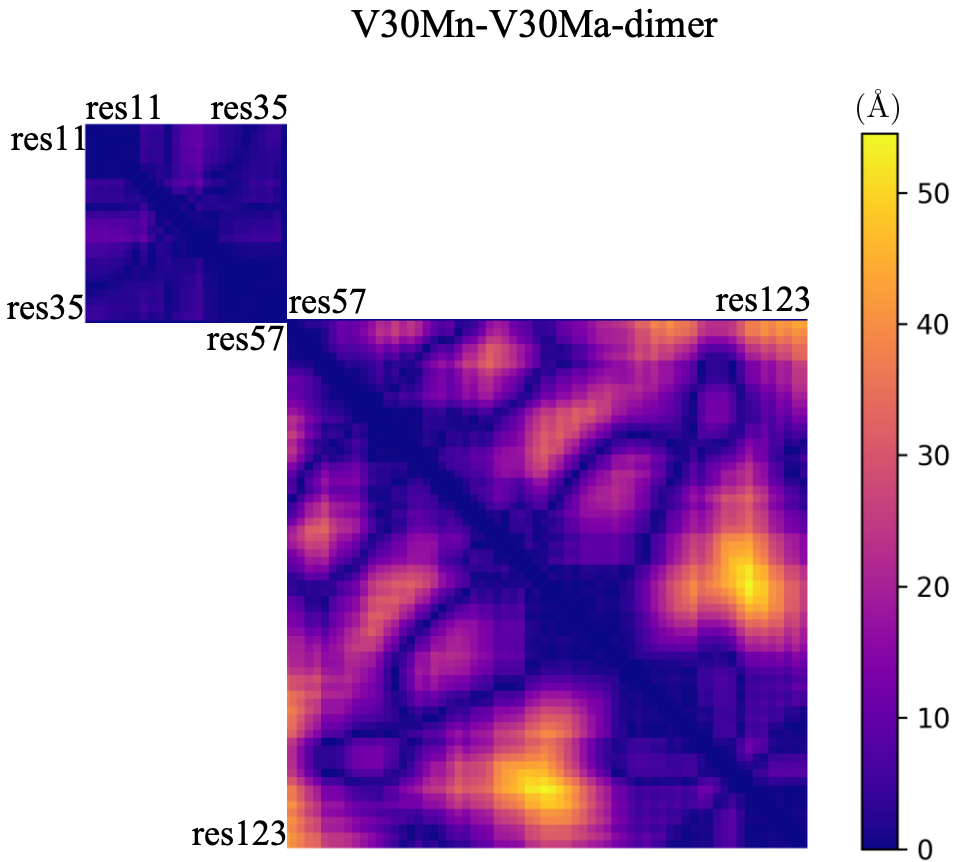}
	\caption{The \textit{truncated hop distance} between chain A of V30Mn and V30Ma-dimer. The entries represented by the rows and columns are the  corresponding residue numbers. Here, we only considered the \textit{truncated hop distance} of AAs of the same sequence fragment. }
	\label{fig: difference of hop distance}
\end{figure}

It is clear that the values near the diagonal are less than 1.5\AA, which illustrates that there could not be stretching and compressing in the process of the formation of amyloid fibril. We suggest that the effect is relatively mild in the formation from chain A of V30Mn to chain A of V30Ma-dimer. {Conversely, differences in hop distance greater than 25\AA are primarily found in three components\ref{fig: morethan25}A. The first and second components correspond to the hop distances between the $\beta_E$-strand (residues 67-73) and the $\beta_F$-strand (residues 88-97), and the hop distances between the EF-loop (residues 82-87) and the GH-loop (residues 113-114), respectively (Fig.\ref{fig: morethan25}B). This is due to the unfolding of the $\alpha$-helix (residues 74-82) and the dissociation of the inter-strand and inter-loop hydrogen bond networks in well-folded TTR. The third component corresponds to the hop distances between the DE-loop (residues 56-66) and the $\beta_H$-strand (residues 115-123). However, this component is absent in some I84S ATTRv structures, specifically I84Sa-p3 and  I84Sa-p4 (Fig.\ref{fig: morethan25}A). The DE-loop forms a gate of the polar channel in the ATTR fibril core and varies in different polymorphs. Most ATTRwt and ATTRv fibril cores prossess a complete shape of the gate (i.e., closed gate and open gate), while the gate is incomplete or absent in some I84S mutated variants (i.e., broken gate and absent gate, Fig.\ref{fig: morethan25}C\cite{nguyen2024structural}.} 

{Moreover, we identified local regions which are sensitive to the conformational transition of TTR by adopting the (truncated) $n$-hop distance. These regions contain the binding sites of Tafamidis, which is approved by the FDA in
2019 for the treatment of ATTR-caused disease. It suggests that the geometrical descriptor we defined gives insight into the drug design for stabilizers to inhibit
the transition to the amyloid fibril state.} More details will be shown in the biology paper \cite{xi2024Top}.
			
\begin{figure}[!ht]
	\centering
	\includegraphics[height=8cm, width=15cm]{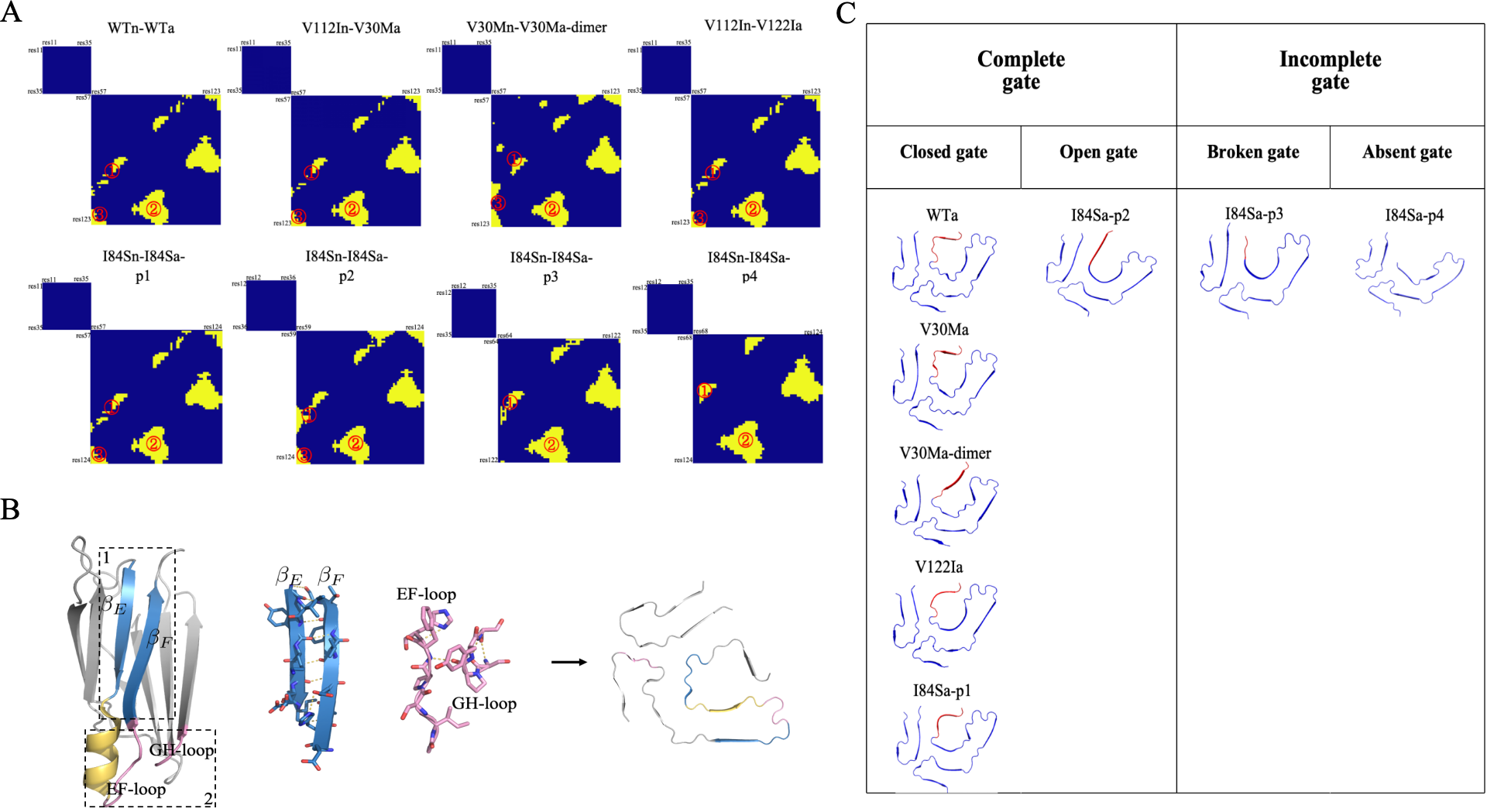}
	\caption{Structural analysis based on the binary matrices with 25\AA cutoff distance. (A) Binary matrices constructed by setting the cutoff distance to 25\AA. (B) The sequence fragments corresponding to the first and second brighter regions and their hydrogen bond interaction. (C) The shape of gate of chain A of ATTR.}
	\label{fig: morethan25}
\end{figure}

\section{Discrete Curvatures and Discrete Torsions}\label{sec:cur-tor}

The discrete curvature and discrete torsion based on the cross ratio and osculating circle was first introduced in \cite{muller2021discrete}, which is used in computer vision to estimate the differences between the background smooth curve and the discrete sample curve. The discrete curvature defined in this way is the discretization of sectional curvature on a smooth curve in a more fashion words. 
In the paper \cite{xi2024Top}, we introduced this important parameter to the analyzing the structure of proteins. We found that compared with the other {geometrical discriptors like backbone dihedral angles, the discrete curvatures and discrete torsions introduced in \cite{carroll2014survey}, discrete curvature resp. discrete torsion is a order two resp. order three parameter which is more sensitive to the weak interactions. Moreover it is still easy to compute although more complicated compared with the one introduced in \cite{carroll2014survey} since we need to solve a $3\times 3$ linear equation.}
In this section, we will give a detailed introduction on these parts.

\subsection{Curvatures and Torsions for Smooth Curves}\label{subsec:smooth-cur-tor}

The curvature and the torsion are two important geometric invariants parametrizing the property of a smooth space curve. {We will first recall the corresponding constructions in the smooth case.}
Let $\gamma(s)\subset \Rds^3$ be a smooth space curve with $s$ the arc length parameter. Recall that the tangent field $T(s)$ is unit for all $s\in \Rds$.
The \textit{principal normal field} is given by $N(s):=\frac{T'(s)}{\parallel T'(s)\parallel}$ and the \textit{binormal field} by $B(s):=N(s)\times T(s)$ while both of them are unit vector fields. It is clear that the three unit vectors $T(s),\ N(s)$ and $B(s)$ are pairwise normal, hence they form a frame at each point.

The \textit{curvature} $\kappa$ is defined as $\kappa:=\parallel T'(s)\parallel$ parametrizing the bending degree of the curve. Geometrically, there exists a space circle passed $\gamma(s)$ which has $N(s)$ as radial vector and $\frac{1}{\kappa}$ as radius. This circle is called \textit{the osculating circle}, which has contact of order at least 2 with $\gamma$ at $\gamma(s)$.
Therefore, they have the common tangent plane which is spanned by $T(s)$ and $N(s)$ and we call it \textit{the osculating plane}.

Since $B(s)$ is an unit field, its differential $B'(s)$ should tangent to $B(s)$. Moreover, it is clear that the following equation holds.
\begin{equation}
    0=\frac{d}{ds}\lbrace B(s),T(s)\rbrace=\lbrace B'(s),T(s)\rbrace+\lbrace B(s),T'(s)\rbrace=\lbrace B'(s),T(s)\rbrace
\end{equation}
Hence, $B'(s)$ should parallel to $N(s)$, the \textit{torsion} of the curve $\gamma$ at $\gamma(s)$ as the number $\tau$ satisfying the property that $B'(s)=-\tau N(s)$.
From all these constructions, the torsion $\tau$ parametrized the speed of a curve moving away from its osculating plane.

\subsection{Basics in quaternions}\label{sec:quaternion}
Let $\Hds$ be the ring of quaternionics which could be treated as a skew field whose elements can be identified with $\Rds\times\Rds^3$.
There exist an isomorphism of Euclidean 3-space $\Rds^3$ with the imaginary parts of $\Hds$. In this paper, we will write $\Hds$ in the following way
\begin{equation}
    \Hds:=\{[r,v]:r\in\Rds,v\in \Rds^3\}
\end{equation}
The first part $r=\Rea q$ resp. the second part $v=\Img q$ of a quaternion $q=[r,v]$ is called the \textit{real part of $q$} resp. the \textit{imaginary part of $q$}. The set of all pure imaginary quaternions is denoted as $\Img \Hds$, i.e., $\Img \Hds:=\{[0,v]:v\in\Rds^3\}$.
Some operators on $\Hds$ under this form are listed in the Table \ref{table: hamquaop}.

\begin{table}
    \begin{tabular}{c c }
        \hline
        Addition       & $[r,v]+[s,w]      = [r+s,v+w] $                                              \\
        Multiplication & $[r,v]\cdot[s,w]  =  [rs-<v,w>,rw+sv+v\times w]$                             \\
        Conjugation    & $\overline{q}     =  [r,-v]$                                                 \\
        Norm           & $|q|              =  \sqrt{q\cdot\bar{q}}=\sqrt{r^2+\parallel v\parallel^2}$ \\
        Inverse        & $q^{-1}              =  \frac{\bar{q}}{|q|^2}$                                 \\
        Cross Ratio    & $\cro(a,b,c,d):=(a-b)(b-c)^{-1}(c-d)(d-a)^{-1}$                              \\
        \hline
    \end{tabular}
    \caption{Some operators on $\Hds$}
    \label{table: hamquaop}
\end{table}

In this paper, the polar representation of a quaternion is quite useful, for a quaternion $q=[r,v]$ its polar representation is defined as $q=|q|[\cos\phi,v \sin\phi]$ with $\parallel v\parallel=1$ and $\phi\in [0,\pi]$. In this case, the square root of $q$ is defined as $\sqrt{q}:=\sqrt{|q|}[\cos\frac{\phi}{2},v\sin\frac{\phi}{2}]$. Note that such defined square root is not unique if $q\in \Rds_{\le 0}$, however, we will not meet this case in our application.

Recall that a \textit{M\"{o}bius transformation} is a homeomorphism of $\Rds^n\cup\{\infty\}$ which is a finite composition of inversions in spheres and reflections in hyperplanes. In other words, any M\"{o}bius transformation $\sigma$ can be written as the form
\begin{equation}
    \sigma(x)=x+\frac{\alpha A(x-a)}{\parallel x-a\parallel^{\epsilon}},\ x\in \Rds^n\cup\{\infty\}
\end{equation}
where $\alpha\in\Rds$, $a,b\in \Rds^n$, $A$ is an orthogonal matrix and $\epsilon$ equals 0 or 2.
It is clear to check that the cross-ratio above is invariant under M\"{o}bius transformation i.e. the following equation holds for all M\"{o}bius transformation $\sigma$
\begin{equation}
    \cro(\sigma(a),\sigma(b),\sigma(c),\sigma(d))=\cro(a,b,c,d)
\end{equation}
Moreover, we have the following Lemma for the properties of the cross ratio.
\begin{lem}\label{lem:cro-ratio}
    Let $a,b,c,d,d'\in \Hds$ be four quaternions, then we have
    \begin{enumerate}
        \item If $d'\in \Hds$, $a\neq c$ and $\cro(a,b,c,d)=\cro(a,b,c,d')$, then we have $d=d'$.
        \item If $a,b,c,d\in \Img \Hds\cong \Rds^3$, then $\cro(a,b,c,d)\in \Rea \Hds$ if and only if $a,b,c,d$ are concyclic.
    \end{enumerate}
\end{lem}
\begin{proof}
    \
    \begin{enumerate}
        \item The first is clear to check.
        \item The second is proved in \cite{bobenko1996discrete}.
    \end{enumerate}
\end{proof}
\begin{lem}\label{lem:imcro}
    Let $a, b, c, d \in \Img\Hds$ be four points not lying on a common circle. Then, the imaginary part of the cross-ratio is the normal of the circumsphere (or plane) at $a$, i.e., for a proper circumsphere with center $m$, we have the imaginary part of $\Img \cro(a, b, c, d)$ is parallel to the vector $(m-a)$ denoted as $\Img \cro(a, b, c, d)\varparallel(m-a)$.
\end{lem}
\begin{rmk}
    If the four points $a, b, c, d$ are coplaner, then $m$ is the point of $\infty$. The vector $\Img \cro(a, b, c, d)$ is orthogonal to the common plane of $a, b, c, d$.
\end{rmk}
\begin{proof}
    This is Lemma 10 of \cite{muller2021discrete}.
\end{proof}

Let us define the `diagonal' points $f(a,b,c,d)$ as
\begin{equation}
    f(a,b,c,d):=((b-a)(c-a)^{-1}\sqrt{\cro(c,a,b,d)}+1)^{-1}((b-a)(c-a)^{-1}\sqrt{\cro(c,a,b,d)}c+b)
\end{equation}
We have the following lemmas
\begin{lem}\label{lem:quaf}
    The `diagonal' point $f(a,b,c,d)$ fullfills the equation that
    \begin{equation}
        \cro(c,a,b,f(a,b,c,d))=-\sqrt{\cro(c,a,b,d)}
    \end{equation}
    Moreover if we assume $a,b,c,d\in \Img(\Hds)\cong \Rds^3$ then the 'diagonal' point satisfies $f(a,b,c,d)\in \Img(\Hds)$. Even more $f(a,b,c,d)$ lies on the circumsphere of $a,b,c,d$.
\end{lem}
\begin{rmk}
    If $a,b,c,d$ are four points in $\Rds^3\cong \Img \Hds$ that form a parallelogram with $ab$ parallel to $cd$, the inserting point $f(a,b,c,d)$ is the intersection of the two diagonals $ac$ and $bd$. This is the reason that it is called a `diagonal' point.
\end{rmk}
\begin{cor}
    The formula $f$ is invariant under M\"{o}bius transformation.
\end{cor}
\begin{proof}
    The Lemma \ref{lem:quaf} is the Lemma 11 and Corollary 4 in \cite{muller2021discrete}. Then from the Lemma \ref{lem:cro-ratio} and Lemma \ref{lem:quaf} the formula $f$ is M\"{o}bius transformation invariant.
\end{proof}

Now we will introduce one of the most important Proposition and also the basis of the further construction.
\begin{prop}\label{prop:four-inserting}
    Let $a,b,c,d\in \Img\Hds\cong \Rds^3$ be four pairwise different points and consider the four inserting points given by $f$ and cyclic permutations
    \begin{eqnarray}
        A=f(d,a,b,c)&\ & B=f(a,b,c,d)\\
        C=f(b,c,d,a)&\ & D=f(c,d,a,b)
    \end{eqnarray}
    Then the four points $A,B,C,D$ are concyclic with
    \begin{equation}\label{eq:crominus1}
        \cro(A,B,C,D)=-1
    \end{equation}
\end{prop}
\begin{proof}
    We only need to prove the Equation (\ref{eq:crominus1}) holds. From Lemma \ref{lem:quaf} the four points $A,B,C$ and $D$ lie on a common sphere i.e. the circumsphere of $a,b,c,d$. Let us fix an unit vector $v_0\in\Rds^3$. It is clear that the space $\Cds_{v_0}:=\{[r,kv_0]:r,v\in\Rds\}$ is isomorphic to the complex plane $\Cds$. Let $\sigma$ is the M\"{o}bius transformation that maps the circumsphere of $a,b,c,d$ to $\Cds_{v_0}$.

    From Lemma \ref{lem:cro-ratio}, that we have $\cro(A,B,C,D)=\cro(\sigma(A),\sigma(B),\sigma(C),\sigma(D))$. Hence without loss of generality, we may assume $a,b,c,d$ are both complex numbers. In this case, the Equation \ref{eq:crominus1} comes from a direct compute, the details are in Theorem 1 of \cite{muller2021discrete}.
\end{proof}
\subsection{Curvatures and Torsions for Discrete Curves}\label{subsec:discrete-cur-tor}
Let us consider the discrete case. Recall that a \textit{discrete curve}is a polygonal curve in $\Rds^2$ or $\Rds^3$ given by its vertices via a map $\gamma:\Zds\to\Rds^3$. We will see the definition of curvature and torsion for a discrete curve follows similar philosophy as their smooth versions. Historically, the curvature for a discrete curve has a long history and is widely used in computer vision and study of complex networks. However, the torsion is pretty new, one is first introduced in \cite{muller2021discrete}. We will introduce them in this section.

Let $\gamma_i$ be a vertex of the discrete curve. From the Proposition \ref{prop:four-inserting}, for each four elements tuple $(\gamma_{i-1},\gamma_{i},\gamma_{i+1},\gamma_{i+2})$ there exist four inserting points (see Figure \ref{fig:dis-torsion}) i.e.
\begin{equation}
    \begin{aligned}
        A_i & = & f(\gamma_{i+2},\gamma_{i-1},\gamma_{i},\gamma_{i+1}),\ B_i & = & f(\gamma_{i-1},\gamma_{i},\gamma_{i+1},\gamma_{i+2}) \\
        C_i & = & f(\gamma_{i},\gamma_{i+1},\gamma_{i+2},\gamma_{i-1}),\ D_i & = & f(\gamma_{i+1},\gamma_{i+2},\gamma_{i-1},\gamma_{i}) \\
    \end{aligned}
\end{equation}
The Lemma \ref{lem:quaf} implies that the four points lie on the circumsphere of $\gamma_{i-1},\gamma_{i},\gamma_{i+1}$ and $\gamma_{i+2}$ which we take it as \textit{the discrete osculating sphere at $\gamma_i$}.
Note that $\gamma_{i-1},\gamma_{i},\gamma_{i+1}$ and $\gamma_{i+2}$ are not concyclic in general, but the Proposition \ref{prop:four-inserting} implies that the four points $A_i$, $B_i$, $C_i$ and $D_i$ are concyclic.

\begin{figure}[ht]
	\centering
	\includegraphics[width=6cm, height=6cm]{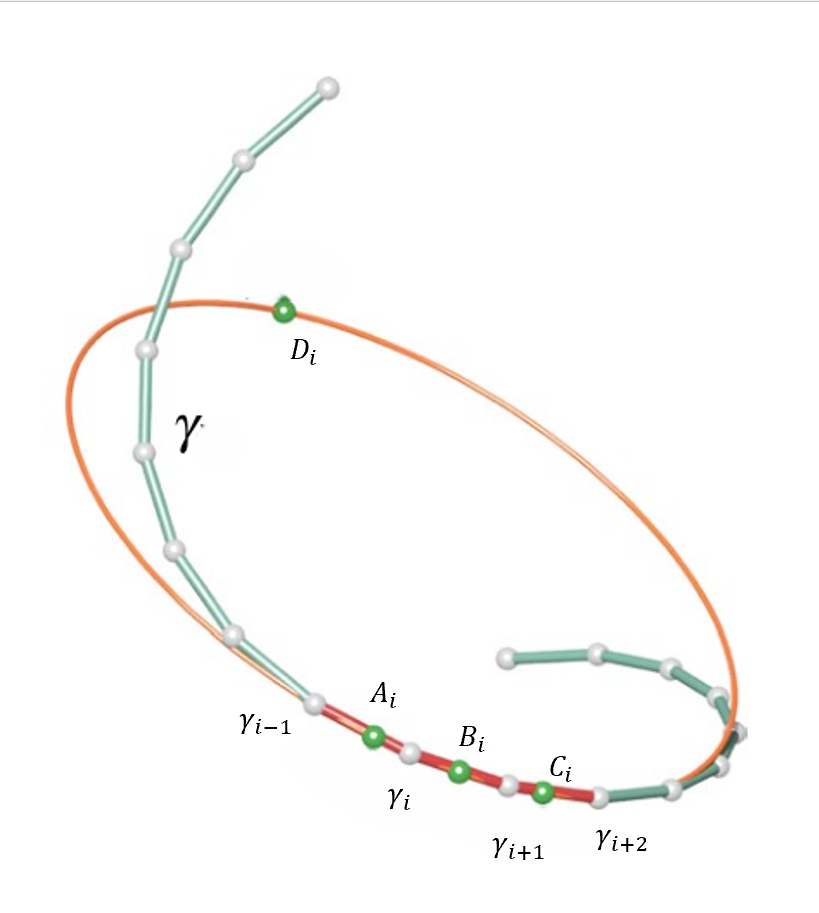}
	\caption{The computation of discrete curvature and discrete torsion where the inserting points were marked with green and the orange circle is the osculating circle at $\gamma_i$. }\label{fig:dis-torsion}
\end{figure}
\begin{Def}\label{def:osccir}
    Let $\gamma$, $\gamma_i$,$A_i$,$B_i$,$C_i$ and $D_i$ as above, the circle denoted by $k_i$ given by $A_i$,$B_i$,$C_i$ and $D_i$ is \textit{the osculating circle of $\gamma$ at $\gamma_i$}. The inverse of the radius of $k_i$ the \textit{discrete curvature at $\gamma_i$} denoted by $\kappa_i$.
\end{Def}

\begin{Def}\label{def:nor-tor}
    \textit{The normal unit vector $N_i$ of $\gamma$ at $\gamma_i$} to be the unit normal vector of $k_i$ at $B_i$. The plane passing through $B_i$ and normal to $N_i$ as \textit{the osculating plane}.
\end{Def}
\begin{Def}
    The discrete torsion of $\gamma$ at $\gamma_i$ is defined as following
    \begin{equation}
        \tau_i:=-\frac{9<\Img\cro(\gamma_{i-1},\gamma_i,\gamma_{i+1},\gamma_{i+2}),N_i>}{2\kappa_i\parallel \gamma_i-\gamma_{i+1}\parallel^2}
    \end{equation}
\end{Def}
\begin{rmk}
    The connection of discrete torsion and smooth version are as following: Let us assume in this Remark that $\gamma(t)$ be a smooth curve and $t$ be any parameter. In this case the torsion of $\gamma$ at $\gamma(t_0)$ could be computed as
    \begin{equation}
        \begin{aligned}
            \tau & = & -\frac{<\gamma'\times\gamma'',\gamma'''>}{\parallel \gamma'\times\gamma''\parallel^2} & = & -\frac{\det(\gamma',\gamma'',\gamma''')}{\parallel \gamma'\times\gamma''\parallel^2} \\
                 & = & \frac{\det(\gamma''',\gamma',\gamma'')}{\parallel \gamma'\times\gamma''\parallel^2}   & = & \frac{<\gamma'\times\gamma''',\gamma''>}{\parallel \gamma'\times\gamma''\parallel^2}
        \end{aligned}
    \end{equation}
    Recall that $\kappa=\frac{\parallel \gamma'\times\gamma''\parallel}{\parallel \gamma'\parallel^3}$ and
    \begin{equation}
        N=B\times T=\frac{\gamma''<\gamma',\gamma'>-\gamma'<\gamma',\gamma''>}{\parallel\gamma'\times\gamma''\parallel\parallel\gamma'\parallel}=\frac{1}{\parallel\gamma'\parallel^2\kappa}\gamma''-\frac{<\gamma',\gamma''>}{\parallel\gamma'\times\gamma''\parallel\parallel\gamma'\parallel}\gamma'
    \end{equation}
    Hence, we have the following
    \begin{equation}
        \tau=\frac{<\gamma'\times\gamma''',N>}{\kappa\parallel\gamma'\parallel^4}
    \end{equation}
    In the discrete case, the role of $\gamma'\times\gamma'''$ is replaced by $\Img\cro(\gamma_{i-1},\gamma_i,\gamma_{i+1},\gamma_{i+2})$. We will show below that such defined discrete torsion satisfies our expectations.
\end{rmk}
Now let us introduce without proof several properties of the definitions above which reveal the fact that the discrete curvature, normal vector and torsion defined above is the approximation of the smooth version.
\begin{prop}\label{Prop:plar-tor-van}
    The discrete torsion vanishes for planer discrete curve.
\end{prop}
\begin{proof}
    If $\gamma$ is a discrete curve lying on the plane $P$, the vector $\Img\cro(\gamma_{i-1},\gamma_i,\gamma_{i+1},\gamma_{i+2})$ is orthogonal to $P$  from the Lemma \ref{lem:imcro}. From the Definition \ref{def:nor-tor}, $N_i\subset P$. Hence the discrete torsion vanish at all points.
\end{proof}

\begin{thm}
    Let $s:\Rds\to \Rds^3$ be a smooth space curve, let $\mu,\epsilon\in \Rds$ be real numbers and let the discrete curve $\gamma:\Zds\to \Rds^3$ with $\gamma_i:=s(\mu+(2i-1)\epsilon)$ be the sample of $s$. Let $\kappa$, resp. $N$ and $\tau$ be the curvature reps. unit normal vector and torsion of $s$ at $s(\mu)$. We have the following properties
    \begin{enumerate}
        \item The discrete curvature is a second approximation of $\kappa_s$, i.e. $\kappa_0=\kappa_s+o(\epsilon^2)$,
        \item The center $m_0$ of the discrete osculating circle $k_0$ of $\gamma$ converges to the center of the smooth osculating circle of $s$ at the same rate, i.e. $m_0-s(\mu)=\frac{N(\mu)}{\kappa_s(u)}+o(\epsilon^2)$,
        \item The discrete Frenet frame $(T_0, N_0, B_0)$ is a second-order approximation of the smooth Frenet frame $(T, N, B)$,
        \item The discrete torsion is a second-order approximation of the smooth torsion i.e. $\tau_0=\tau+o(\epsilon^2)$
    \end{enumerate}
\end{thm}
\begin{proof}
    This is the Theorem 4, Lemma 14, Lemma 15 and Theorem 5 in \cite{muller2021discrete}. The main points is similar as the proof of the Proposition \ref{prop:four-inserting}, we first map all points to a complex plane by a M\"{o}bius transformation. Then we could check by direct computation.
\end{proof}

\subsection{Applications to Analysis of the Structure of Amyloid fibrils}
We have shown a way of treating a peptide chain as a discrete curve and  its application in the Section \ref{sec:hop-dis}. In this section, we considered another way as following:
Recall that from the Figure \ref{fig: peptide chain}, the backbone consists of amino nitrogens, $\alpha$-carbons, carbonyl carbons and carbonyl oxygens. We treated the the coordinates of amino nitrogens, $\text{C}_{\alpha}$s, and carbonyl carbons on the backbone as the vertices of a discrete curve in $\Rds^3$. 

For a fixed layer of ATTR, we may assume its layer index in the protein is $\ast$, let $\text{N}^{\ast}_i$ (resp. $\text{C}^{\ast}_{\alpha_i}$ and $\text{C}^{\ast}_i$) be the coordinates of amino nitrogens (resp. alpha carbons and carbonyl carbons) with $i\in \Zds^{\ge 1}$ being the index of amino acids (AAs) in the layer. Therefore, the curvature and torsion at each atom could be derived from the position of itself, one atom ahead and two atom after, e.g. $\text{C}^{\ast}_{\alpha_i}$, $\text{C}^{\ast}_i$, $\text{N}^{\ast}_{i+1}$ and $\text{C}^{\ast}_{\alpha_{i+1}}$ determine the curvature and the torsion at $\text{C}^{\ast}_i$. We computed the absolute values of the curvatures ($|\kappa|$) and torsions ($|\tau|$) at different atoms from all residues (the curvatures resp. torsions and atoms involved in the computation are listed in \ref{eq:cur-tor-atom}) in each ATTR fibril structure, as well as the arithmetic mean of absolute values of the curvatures ($\overline{|\kappa|}$) and torsions ($\overline{|\tau|}$), and the (unbiased) variance of $|\kappa|$ ($S^2(|\kappa|)$) and $|\tau|$ ($S^2(|\tau|)$), respectively.

\begin{equation}\label{eq:cur-tor-atom}
    \begin{aligned}
        \text{C}^{\ast}_{i-1},\text{N}^{\ast}_i,\text{C}^{\ast}_{\alpha_i},\text{C}^{\ast}_i     & \to & \hbox{Curvatures and Torsions at } \text{N}_i          \\
        \text{N}^{\ast}_i,\text{C}^{\ast}_{\alpha_i},\text{C}^{\ast}_i,\text{N}^{\ast}_{i+1}     & \to & \hbox{Curvatures and Torsions at } \text{C}_{\alpha_i} \\
        \text{C}^{\ast}_{\alpha_i},\text{C}^{\ast}_i,\text{C}^{\ast}_{i+1},\text{C}^{\ast}_{i+1} & \to & \hbox{Curvatures and Torsions at } \text{C}_{i}        \\
    \end{aligned}
\end{equation}
The pseudo-code of computing curvature and torsion is shown as the Algorithm \ref{alg:cur-tor}.{To solve the curvature and torsion at one atom, one needs to solve a $3\times 3$ linear equation for determining the center of the osculating circle. The computation complexity is a constant. Therefore the computation complexity for computing all discrete curvatures and torsions is linear with respect to the number of atoms. Moreover the computation could be conduct parallelly for large size of data set.}
\IncMargin{1em}
\begin{algorithm}
    \SetKwData{Left}{left}\SetKwData{This}{this}\SetKwData{Up}{up}
    \SetKwData{Compute}{compute}
    \SetKwFunction{Union}{Union}\SetKwFunction{FindCompress}{FindCompress}
    \SetKwInOut{Input}{input}\SetKwInOut{Output}{output}

    \Input{A series of shape $3l\times 3$ consists of coordinates of all amino nitrogens, $\alpha$-carbons and carbonyl carbons}
    \Output{A series of shape $3l\times 2$ consists of curvatures and torsions at each atom}
    \BlankLine
    \For{$i\leftarrow 1$ \KwTo $l-3$}{

        \Compute inserting points $A,B,C,D$

        \Compute circumcircle center $O$ by solving equations \begin{eqnarray*}
            (X-\frac{A+B}{2})\overrightarrow{BA}&=&0\\
            (X-\frac{A+C}{2})\overrightarrow{CA}&=&0\\
            (\overrightarrow{AB}\times\overrightarrow{AC})\overrightarrow{XA}&=&0
        \end{eqnarray*}

        \Compute $\kappa=\frac{1}{\parallel\overrightarrow{AO}\parallel}$

        \Compute $N=\frac{\overrightarrow{AO}}{\parallel\overrightarrow{AO}\parallel}$

        \Compute torsion $\tau$

    }
    \caption{Algorithm of computing curvature and torsion}\label{alg:cur-tor}
\end{algorithm}\DecMargin{1em}

Let $(\ast-1)$ and $(\ast+1)$ be two nearby chains of the $\ast$-th chain. The backbone hydrogen bonds were always between accepter, the carbonyl oxygen of the $i$-th AA of the chain $\ast$ (denoted by $\text{O}_i^{\ast}$), and doner, the amino nitrogen of $(i+1)$-th AA of the chain $(\ast-1)$ or the chain $(\ast+1)$ (denoted by $\text{N}_{i+1}^{(\ast-1)}$ or $\text{N}_{i+1}^{(\ast+1)}$, respectively), according to the direction of the carbonyl. Let $\tilde{d}_{i,+}^{\ast}$ (resp. $\tilde{d}_{i,-}^{\ast}$) be the distance between $\text{O}_i^{\ast}$ and $\text{N}_{i+1}^{(\ast+1)}$ (resp. the distance between $\text{O}_i^{\ast}$ and $\text{N}_{i+1}^{(\ast-1)}$). The squared distance differences at the $i$-th AA of the chain $\ast$ is defined as,
\begin{equation}\label{eq:sq-dis-layers}
	\tilde{d}_i^{\ast}:=|(\tilde{d}_{i,-}^{\ast})^2-(\tilde{d}_{i,+}^{\ast})^2|
\end{equation} (See Fig.\ref{fig:computation-cur-tor-hydro-distance}).
\begin{figure}[ht]
    \centering
    \includegraphics[width=\textwidth]{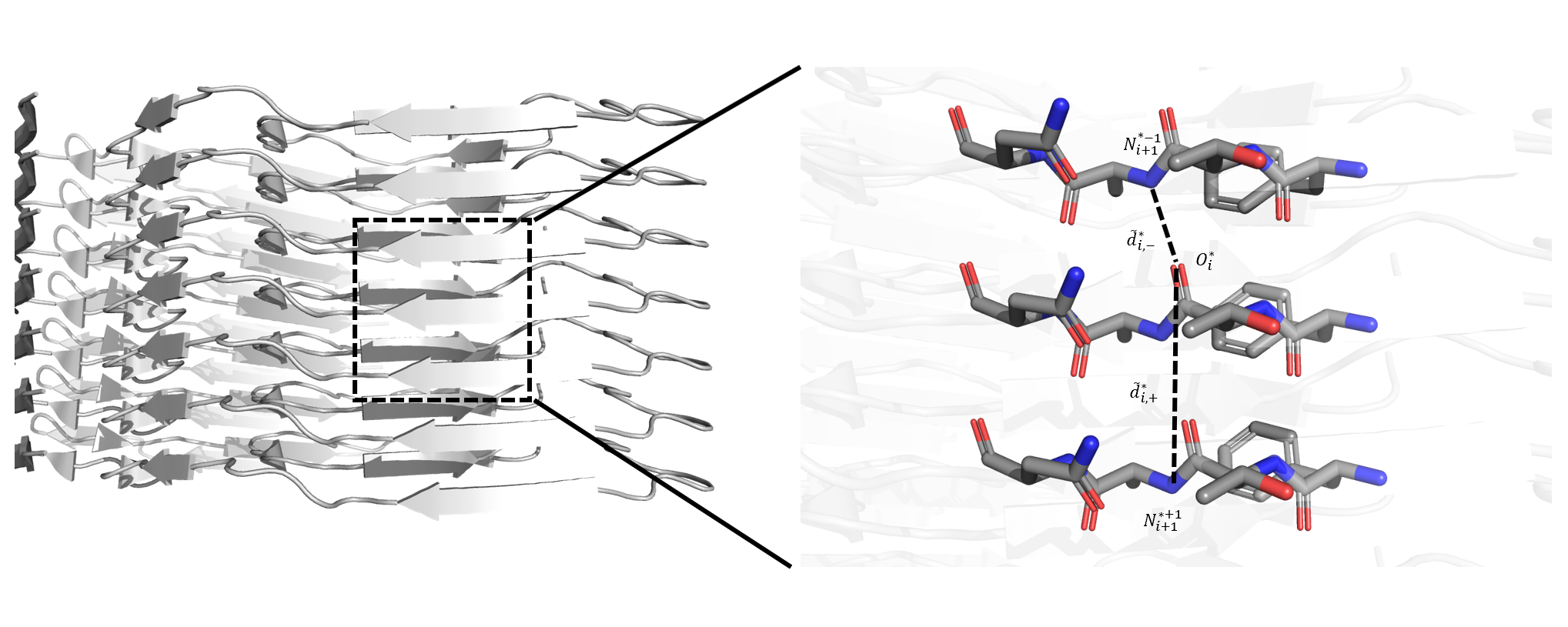}
    \caption{The squared distance differences at the $i$-th AA of the chain B.
    The atoms are marked in different colors, i.e., carbons marked with grey, nitrogens marked with blue and oxygens marked with red.}
    \label{fig:computation-cur-tor-hydro-distance}
\end{figure}

\begin{table}[!h]
    \centering
    \caption{\label{table:Ave-Var-Cur-Tor}The mean and variance of absolute values of curvatures and torsions for V30Ma-dimer.}
    \resizebox{1\textwidth}{!}{
		\centering
    \begin{tabular}{c c c c c c c c c c c c c c c c c}
        \hline
             & \textbf{$\bar{\kappa}$} & \textbf{$\bar{\kappa}_\Ntxt$} & \textbf{$\bar{\kappa}_{\Ctxt\alpha}$} & \textbf{$\bar{\kappa}_\Ctxt$} & \textbf{$\bar{\tau}$} & \textbf{$\bar{\tau}_\Ntxt$} & \textbf{$\bar{\tau}_{\Ctxt\alpha}$} & \textbf{$\bar{\tau}_\Ctxt$} &\textbf{$S^2(\kappa)$} & \textbf{$S^2(\kappa_\Ntxt)$} & \textbf{$S^2(\kappa_{\Ctxt\alpha})$} & \textbf{$S^2(\kappa_\Ctxt)$} & \textbf{$S^2(\tau)$} & \textbf{$S^2(\tau_\Ntxt)$} & \textbf{$S^2(\tau_{\Ctxt\alpha})$} & \textbf{$S^2(\tau_\Ctxt)$}\\
        \hline
        V30Ma-dimer & 0.305                   & 0.445                         & 0.421                                 & 0.050                         & 5.048                 & 1.521                       & 1.858                               & 11.766 &0.052                  & 0.024                        & 0.034                                & 0.000                        & 64.589               & 1.441                      & 9.727                              & 114.901                     \\
        \hline
    \end{tabular}
    }
\end{table}

We conducted the computation to several amyloid fibrils. We found that the torsions at carbonyl carbons are significantly larger than the one in other places (we list one example V30Ma-dimer in Figure \ref{fig:tor-dist-scatter} (A) and Table \ref{table:Ave-Var-Cur-Tor}). We showed that this anomalies reflect slight affects of the layer-layer hydrogen bonds.

Note that the carbonyls will form hydrogen bonds with amino hydrogen in the nearby chain along fibril axis, which play important roles in the stabilization of amyloid fibrils. We suggested that this anomaly in torsion reflects the existence and power of hydrogen bonds. We made the correlation analysis for each ATTR by using the linear model between the absolute value of torsions at carbonyl carbons in the $i$-th AA at the $\ast$-th layer ($\tau_i^{\ast}$) and the squared distance differences ($\tilde{d}_i^{\ast}$), where $\ast$ runs over all possible layers and $i$ run over all AAs in the layer (the case V30M-dimer is showed in Fig.\ref{fig:tor-dist-scatter} (B)). The zero assumption is listed as following,
\begin{equation}
    H_0:\hbox{The absolute value of torsion } |\tau_i^{\ast}|  \hbox{has non-negative correlation with } \tilde{d}_i^{\ast}
\end{equation} 
Note that the $p$-value is computed using the Wald Test with the $t$-distribution of the test statistic and S.D stands for the standard error.

We found the $p$-value of $H_0$ is rather small (We show one example V30Ma-dimer in Table \ref{table:cor-tor-dis}). Therefore, we suggested that $|\tau_i^{\ast}|$ is {negatively} correlated with $\tilde{d}_i^{\ast}$. Since the strength of hydrogen bonds is {negatively} related with the square of the distance, we could derive from all these facts that the above anomaly of the torsion is caused by the layer-to-layer hydrogen bonds.
\begin{figure}[!ht]
	\sidesubfloat[]{
		\begin{minipage}[t]{0.4\linewidth}
			\centering
			\includegraphics[width=\linewidth]{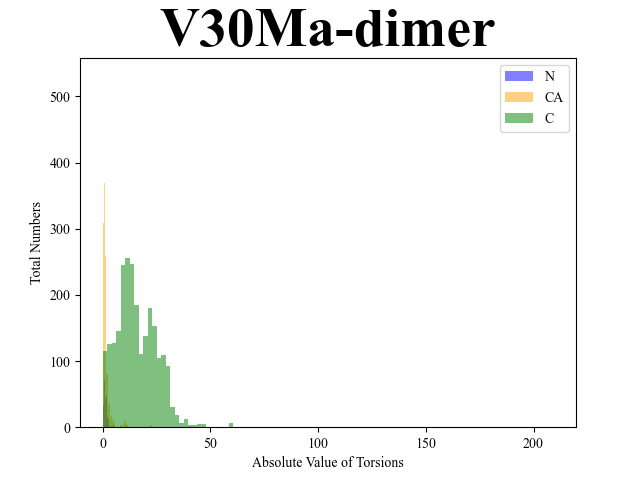}
		\end{minipage}
	}
	\hfill
	\sidesubfloat[]{
		\begin{minipage}[t]{0.4\linewidth}
			\centering
			\includegraphics[width=\linewidth]{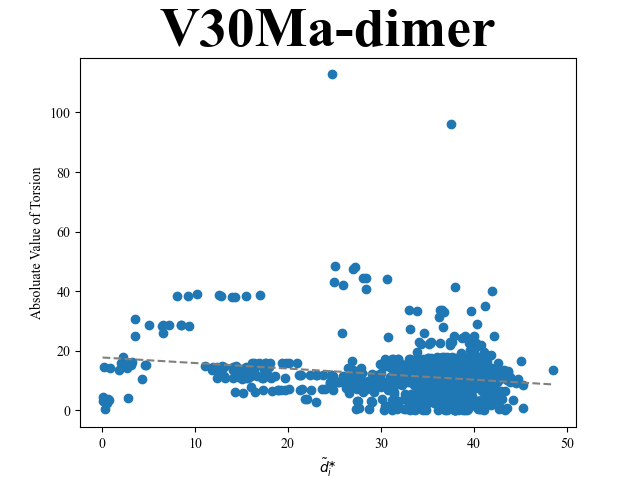}
		\end{minipage}
	}

	\caption{{(A) Distributions of absolute values of Torsions for V30Ma-dimer. The $x$-axis represents the absolute value of torsion and the $y$-axis represents the number of atoms. We take purple bar for nitrogens, orange for alpha carbons and green for carbonyl carbons. The discrete torsion is dimensionless. (B)The scatter of absolute values of torsions at carbonyl carbons $\tau_i^{\ast}$ against the squared distance differences $\tilde{d}_i^{\ast}$ between nearby layers
    for all amyloid fibrils of V30Ma-dimer. Recall that $\tilde{d}_i^{\ast}$ is defined in Equation \ref{eq:sq-dis-layers} where $\ast$ stands for the index of the layer in the protein and $i$ stands for the index of AA in the layer. The horizontal axis stands for the value  $\tilde{d}_i^{\ast}$, the vertical axis stands for $|\tau_i^{\ast}|$ and the dashed grey line is the result of regression.}}
	\label{fig:tor-dist-scatter}	
\end{figure}



\begin{table}[ht]
    \centering
    \caption{\label{table:cor-tor-dis}The correlation of $|\tau_i|$ and $|(\tilde{d}^i_{\cdot,-})^2-(\tilde{d}^i_{\cdot,+})^2|$ for V30Ma-dimer. S.D stands for the standard error}
    \begin{tabular}{c c c c c c c}
        \hline
             & \textbf{Slope} & \textbf{Intercept} & \textbf{Pearson} & \textbf{$p$-value}    & \textbf{S.D slope} & \textbf{S.D intercept} \\
        \hline
        V30Ma-dimer & -0.187         & 17.765             & -0.194           & $2.325\times 10^{-9}$ & 0.032              & 1.094                  \\
        \hline
    \end{tabular}
\end{table}

\section{Topological Data Analysis}\label{sec:Top-Methods}


\subsection{Persistent Homology in a Nutshell}
We are going to introduce the definition of PH first in a general setting. Let us take $\Kds$ be a fixed ring. Let $X$ be a set and $K$ be a simplicial complex constructed from $X$ which is a collection of non-empty finite subsets of $X$, such that for every set $\sigma$ in $K$ and every subset $\sigma'\subset \sigma$, the set $\sigma'$ also belongs to $K$.
We will call $\sigma\in K$ a simplex.

Suppose further that each simplex $\sigma$ is oriented, there exists a chain complex
\begin{equation}\label{eq:chain-complex}
	\cdots\to C_{n+1}(K) \xrightarrow{\partial_{n+1}} C_{n}(K)\to\cdots
\end{equation}
where $C_n(K)$ is the free $\Kds$-module generated by the basis of simplicies $\sigma=[x_0,\cdots,x_n]$ in $K$ of length $n+1$ and the boundary map $\partial_{n+1}:C_{n+1}(K)\to C_{n}(K)$ is a $\Kds$-linear map defined in a natural way.

\begin{Def}
	The filtration of $K$ is a nested sequence of subcomplexes $\{K_{t}\}_{t\in T}$ of $K$ with $T\subset \Rds$ such that if $t\leq t'$ then $K_{t}\subset K_{t'}$ and $\cup_{t\in T}K_{t}=K$. Note that the index set $T$ could be finite or infinite.

\end{Def}
Suppose $t\leq t'$ be two different real numbers, the inclusion map $K_{t}\hookrightarrow K_{t'}$ is also a map of chain complexes namely the following diagram commutes
\begin{equation}
	\xymatrix{\cdots\ar[r]&C_{n+1}(K_t)\ar[r]^{\partial_{n+1}^t}\ar[d]&C_{n}(K_t)\ar[d]\ar[r]^{\partial_n^t}&C_{n-1}(K_t)\ar[d]\ar[r]&\cdots\\
	\cdots\ar[r]&C_{n+1}(K_{t'})\ar[r]^{\partial_{n+1}^{t'}}&C_{n}(K_{t'})\ar[r]^{\partial_n^{t'}}&C_{n-1}(K_{t'})\ar[r]&\cdots}
\end{equation}
where $\partial_{n}^t$ is the restriction of $\partial_{n}$ to $C_n(K_t)$. Hence, it induces a map of homology group $H_{n}(K_{t})\to H_{n}(K_{t'})$.
\begin{Def}\label{def:persistent homology}
	The $(t,t')$-persistent $n$-th homology group is defined as the image of the induced map $H_{n}(K_{t})\to H_{n}(K_{t'})$.
	In other words, it is the following group
	\begin{equation}
		H_n^{t\to t'}=Z_n^t/(B_n^{t'}\cap Z_n^t),
	\end{equation}
	where $Z_n^t$ is the $n$-th cycle group of $K_t$, i.e., $\ker\partial_{n}^t$, and $B_n^{t'}$ is the $n$-th boundary group of $K_{t'}$, i.e., $\Img\partial_{n+1}^{t'}$. The rank of $H_n^{t\to t'}$ is called the $(t,t')$-persistent $n$-th Betti number denoted as $\beta_n^{t\to t'}$ \cite{zomorodian2004computing}.
\end{Def}
The value $t_0$ where for every $\delta>0$ the induced map $H_{n}(K_{t_0-\delta})\to H_{n}(K_{t_0+\delta})$ is not an isomorphism is called a \textit{homological critical value}.
\begin{rmk}
	In the real world applications, people always assume the coefficient ring $\Kds$ to be $\Zds/2\Zds$ in which case the orientation of the simplex is not important, since in such circumstances the boundary of the simplex is not related to the orientation. Hence we will usually omit the choice of orientation in the TDA process.
\end{rmk}
\begin{Def}
	A persistence $\Kds$ module $M$ is a family of $\Kds$-modules $M_i$, together with homomorphisms $\varphi_i : M_i \to M_{i+1}$. We say a persistence module $M$ of finite type if each component module is a finitely generated $\Kds$-module, and if there exist positive integer $m_0$ such that the maps $\varphi_i$ are isomorphisms for $i \ge m_0$ \cite{zomorodian2004computing}.
\end{Def}
\begin{rmk}
	There exists a functor from the category of persistence modules of finite type over $\Kds$ to the category of finitely generated non-negatively graded modules over the polynomial ring $\Kds[x]$, namely take the grade on $M$ as the natural one and $x$ acts on $M$ via the map $x\to \varphi$. From the Artin-Rees theory \cite{eisenbud2013commutative}, this construction turns out to be an equivalence of categories.
\end{rmk}
From our requirements, from now on let us assume in additional that $\Kds$ is a field and
the above simplicial complex $K$ is finite. This means each group in chain complex \ref{eq:chain-complex} is a finite dimensional vector space over $\Kds$. In this case, the `filtration' on the chain complex \ref{eq:chain-complex}, which will be called a persistent complex \cite{zomorodian2004computing}, is actually of finite length. In other words, the total chain complex $C_{\bullet}(K)$ and the homology module $H_{\bullet}(C_{\bullet}(K),\Kds)$ are both persistent module of finite type. The total chain complex $C_{\bullet}(K)$ is not only a graded $\Kds$-module but also a $\Kds[x]$-module with the variable $x$ of degree 1 acts as shift map via the filtration. The total homology module $H_{\bullet}(C_{\bullet}(K),\Kds)$ retains this $\Kds[x]$-module structure. And from the structure theorem for principal ideal domains (PID), note that $\Kds[x]$ is a PID for a field $\Kds$, we have the following decomposition
\begin{equation}\label{eq:permod-decom}
	H_{\bullet}(C_{\bullet}(K),\Kds)\cong \mathop{\oplus}\limits_{i}x^{t_i}(\Kds[x])\oplus(\mathop{\oplus}\limits_{j}x^{r_j}(\Kds[x]/(x^{s_j}\Kds[x])))
\end{equation}
Here $t_i$, $r_j$ and $s_j$ are all integers represent index of filtration. The decomposition theorem has the following meaning: the free parts of Equation \ref{eq:permod-decom} are in bijective correspondence with those homology generators birth at parameter $t_i$ and persist for all future parameter values. The torsion parts correspond to those homology generators birth at parameter $r_j$ and death at parameter $r_j + s_j$.


PD is one of the common visualization methods of the PH. PD is a multiset of points in the extended two-dimensional plane $\overline{ \Rds}^2$ where $\overline{ \Rds} = \Rds \cup \{\infty\}$, whose coordinates are respectively correspond to the birth and death time of the homology generators, union the diagonal points \cite{edelsbrunner2008persistent}.
Given two PDs $P_1$ and $P_2$, there are two commonly used metrics to compare them, that is, the \emph{Wasserstein-$q$ distance} $W_q$ and \emph{Bottleneck distance} $W_\infty$. They are defined as following:
\begin{equation}
	W_q(P_1, P_2)= (\inf\limits_{\eta : P_1 \rightarrow P_2}\sum\limits_{p\in P_1} ||p-\eta(p)||_\infty^q)^{\frac{1}{q}},  %
\end{equation}

\begin{equation}
	W_\infty(P_1, P_2)= \inf\limits_{\eta : P_1 \rightarrow P_2}\sup\limits_{p\in P_1} ||p-\eta(p)||_\infty,
\end{equation}
where $\eta$ runs over all the bijections between the points in the two diagrams. The \emph{Wasserstein-$q$ distance} is capable of capturing local differences in the PDs, while the \emph{Bottleneck distance} is more concerned with relatively global differences.

\subsection{The Vietoris-Rips (VR) complex}\label{subsec:tda}
For different types of data, there are various complexes to characterize the data, such as (1) for point cloud data (PCD), Alpha complex, VR complex, \v{C}ech complex and Witness complex are commonly used;  (2) for imaging data, cubical complex may be a good choice; (3) for graph or complex network data, path complex may perform better. By treating a layer of amyloid fibril as a discrete curve, there are several candidate complexes to characterize this PCD. Here, we adopted the VR complex since it is precise and easy to compute.

Let $X$ be a finite point set with a index in $n$-dimensional Euclidean space $(\Rds^n,d)$, where $d$ is the natural metric on the Euclidean space and $\Kds$ be the field $\Zds/2\Zds$.
\begin{Def}
	The Vietoris-Rips (VR) complex (denoted by $\VR(X)$) of $X$ with parameter $\epsilon$ is the set of all simplices $\sigma=\{x_0,\cdots, x_k\}$, such that $d(x_i, x_j) \leq 2\epsilon$ for all $(i, j)$ \cite{cang2018integration}. 	
\end{Def}
As we explained before, the orientation of the simplex could be omitted.
And there exist a natural filtration on the Vietoris-Rips (VR) complex given by the parameter $\epsilon$. Since the number of $X$ is finite, this filtration is always of finite length, i.e.,
\begin{equation}\label{eq: filtration}
	X=\VR(X)_0\subseteq{\VR(X)_1}\subseteq{\cdots}\subseteq{\VR(X)_m}
\end{equation}
We will modify our notations a bit as follows, let $i$ and $p$ be integers
\begin{equation}\label{persistent homology}
	H_n^{i,p}:=H_n^{i\to i+p}=Z_n^i/(B_n^{i+p}\cap Z_n^i),
\end{equation}
And we have similar modification for the notations of Betti number $\beta_k^{i,p}:=\beta_k^{i\to i+p}$.

Figure \ref{fig:pd7ob4} is the PD of the C-terminal fragment of chain A of V30MA-DIMER using the VR complex. The data set is given by taking the coordinates of $C_{\alpha}\text{s}$ on the C-terminal fragment. Since each chain of V30MA-DIMER is almost planner, we only computed the persistence on $H_0$ and $H_1$. Red points and blue points represent the 0-dimensional and 1-dimensional homological generators respectively in the Fig.\ref{fig:pd7ob4} and Fig.\ref{fig:tda}A.
\begin{figure}[ht]
	\centering
	\includegraphics[width=8cm, height=6cm]{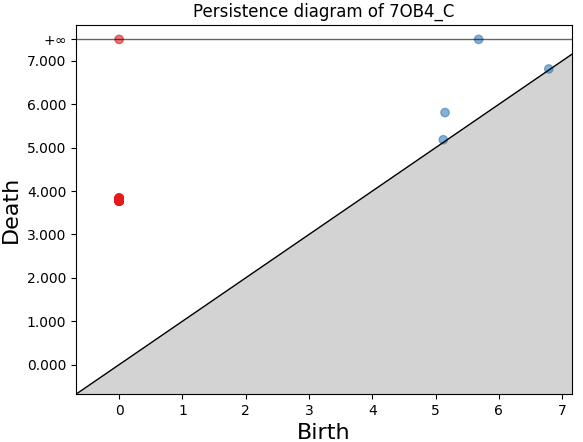}
	\caption{The PD of the C-terminal fragment of V30MA-DIMER. Red points and blue points represent the 0-dimensional and 1-dimensional homological generators, respectively. The unit is the Angstrom (\AA). }\label{fig:pd7ob4}
\end{figure}

\subsection{Applications to Analysis of the Structure of Amyloid fibrils}
The pathological function of amyloid fibril is largely determined by its structure. It is important to discriminate the structure of amyloid fibrils, followed by exploring the functions which are closely related to those structures.  In biology, the RMSD is a common indicator to quantify the difference between structures after aligning. Let $\text{Prot}_{\text{1}}=\{c_1, \dots, c_n\}$ and $\text{Prot}_{\text{2}}=\{c_1', \dots, c_n'\}$ be the coordinates of $C_{\alpha}\text{s}$ of two proteins, respectively. The RMSD between $\text{Prot}_{\text{1}}$ and $\text{Prot}_{\text{2}}$ is defined as
\begin{equation}
	\text{RMSD} = \sqrt{\frac{\sum\limits_{i=1}^{n}||c_i - c_i'||^2}{n}}.
\end{equation}

It is clear that the RMSD only considers the difference of pairwise distance (Fig.\ref{fig:tda}B). Here, we adopted the PH to quantify the difference between two amyloid fibril structures from multiscales (Fig.\ref{fig:tda}A). The pseudo-code of comparing two amyloid fibrils is shown as the Algorithm \ref{alg:TDA}. {Please contact us if you require the source code.} The results are displayed in the Fig.\ref{fig:tda}C.


\begin{figure}[!ht]
	\centering 
	\includegraphics[width=12cm]{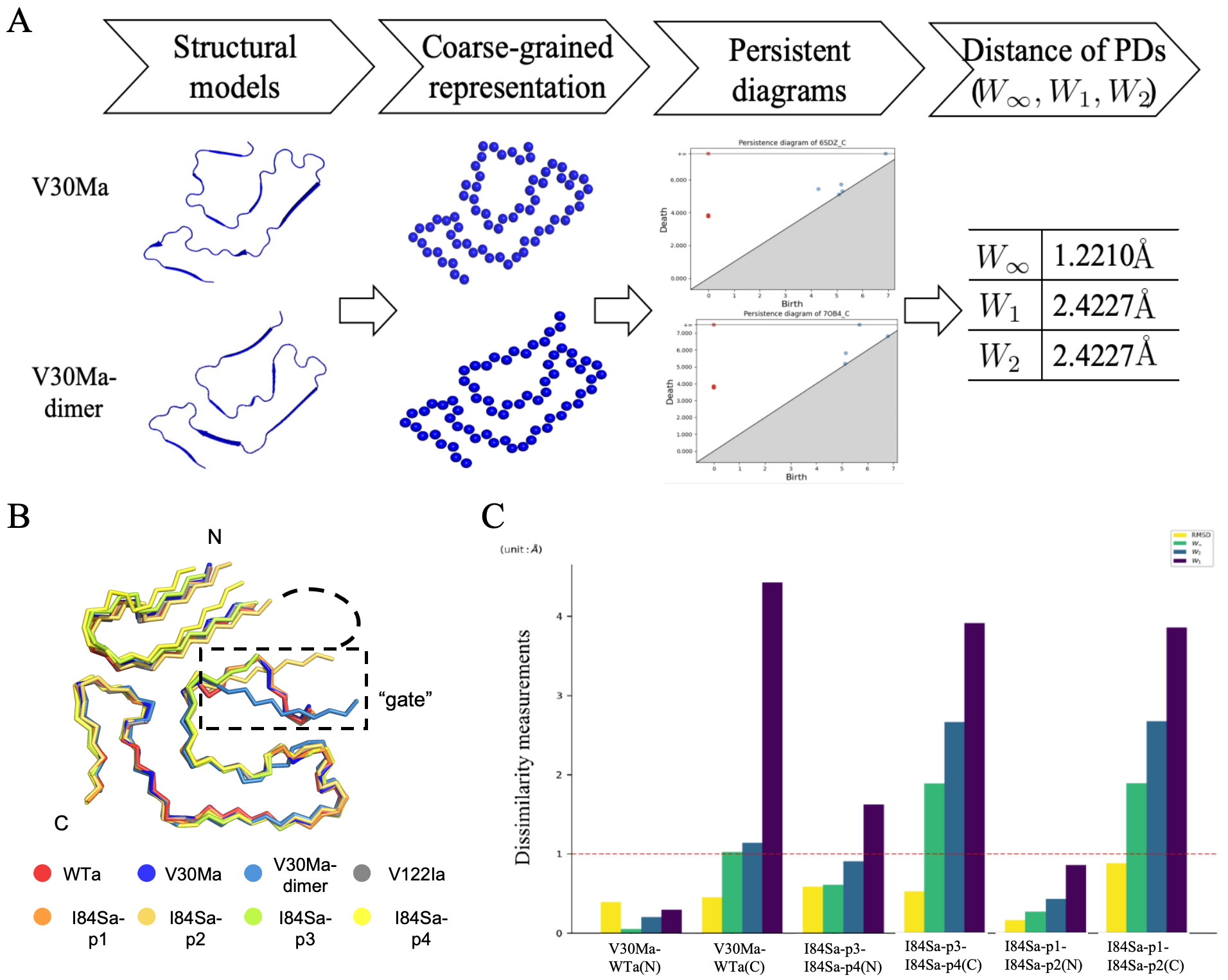}
	\caption{{Dissimilarity measurements. (A) Flowchart of capturing dissimilaritues between the C-terminals on chain A of V30Ma and V30Ma-dimer by adopting TDA. (B) The alignment between chain A of ATTRs. (C) Dissimilarity measurements based on various methods. The capital letters N and C represent the terminal fragments of chain A of each ATTR, respectively.}}
	\label{fig:tda}
\end{figure}
We found that the RMSDs of two fragments are both less than 1\AA. It indicates that the structure of {ATTRs} are considered as almost the same in the RMSD context. However, for example the patient carries 6SDZ and the patient carries 8ADE have different clinical manifestations \cite{steinebrei2022cryo}, which implies that there should be differences between the structure of these amyloid fibrils. We find that the global difference between these amyloid fibrils can be captured by the $W_\infty$ distance, while the local difference between these amyloid fibrils is scaled up as the exponent order $q$ decreases. PH considers the higher-order relations among these $C_{\alpha}\text{s}$, which could give us more information to distinguish the structure of amyloid fibrils (Fig.\ref{fig:tda}C).

\IncMargin{1em}
\begin{algorithm}
	\SetKwData{Left}{left}\SetKwData{This}{this}\SetKwData{Up}{up}
	\SetKwData{Compute}{compute}
	\SetKwFunction{Union}{Union}\SetKwFunction{FindCompress}{FindCompress}
	\SetKwInOut{Input}{input}\SetKwInOut{Output}{output}

	\Input{The coordinates of atoms of two fibril structures $\text{Prot}_{\text{1}}$  and $\text{Prot}_{\text{2}}$ , the maximal filtration value $l$, the exponent order $q$}
	\Output{Distances $\{d_\infty, d_q\}$ to quantify the difference between the two fibril structures}
	\BlankLine
	\For{$i\leftarrow 1$ \KwTo $2$}{
	$f_i \leftarrow$  the filtration of $P_i$ based on a suitable rule, that is,  $f_i=\{v_j\}_{j=1}^n$ where $v_n \le l$

	$\{K^i_j\}_{j=1}^n  \leftarrow$ a nested sequence of simplicial complexes based on $f_i$

	$\text{PD}_{i}  \leftarrow$ the persistent homology analysis of $\{K^i_j\}_{j=1}^n$
	}

	$d_\infty \leftarrow W_\infty(\text{PD}_1, \text{PD}_2)$

	$d_q \leftarrow W_q(\text{PD}_1, \text{PD}_2)$
	\caption{Algorithm of using TDA to compare fibril structures}\label{alg:TDA}
\end{algorithm}\DecMargin{1em}

\section{Discussion}\label{sec: discussion}
{In this study, we introduced geometrical and topological tools for analysing structures of amyloid fibrils}. Geometrically, we defined the (truncated) $n$-hop distance to analyze the potential conformational transformations during the formation of ATTR fibrils from TTR tetramers. We found that the disruption of the $\alpha$-helix alters the interactions between the original secondary structures, which affected the values of hop distance significantly.  
Moreover, based on the special  ``pseudo-2D'' structure of amyloid fibril, we defined a novel discrete curvature and torsion for structural analysis of intra- and inter- layers of ATTR. The torsions at the carbonyl carbons are significantly greater than those at other backbone atoms, and the absolute value of torsions is negatively correlated with the squared distance differences between the atoms on nearby layers that form the backbone hydrogen bonds. We concluded that the anomalies in the torsions of carbonyl carbons reflect slight effects of the layer-layer hydrogen bonds, the widest range of interactions in cross-$\beta$ structures. It is still unclear which specific physical properties of amyloid fibrils are relevant to the torsions. More datasets and experiments are needed to establish the link between these geometrical trends and macroscopic physical properties.

Additionally, we compared the differences between various polymorphs of ATTR adopting topological analysis. Compared to RMSD, TDA can more clearly distinguish the local differences between polymorphs and is expected to be useful for classifying different fibrillar polymorphs formed by the same amyloid protein. Unlike TTR polymorphs, which have low overall differences (small RMSDs), many amyloid proteins can form significantly different polymorphs in different contexts. {However, the application of TDA to the analysis of these polymorphs is still under development due to the deficiency of structures of amyloid fibrils. Because of the small acount of data, the superiority of TDA has not clearly stand out compared with other techniques of  data dimensionality reduction, such as PCA can not distinguish I84S-P1 and I84S-P2(Fig.\ref{fig: pca}), but the others can be  distinguished. Whereas, TDA has advantages that these methods do not, such as the multiscale characterization of structures of amyloid fibrils and the interpretative of the extracted structural features, we believe that the combination of TDA with  machine learnings could give novel insights in the research of amyloid fibrils in the future. }

\begin{figure}
	\centering
	\includegraphics[height=3cm, width=4.5cm]{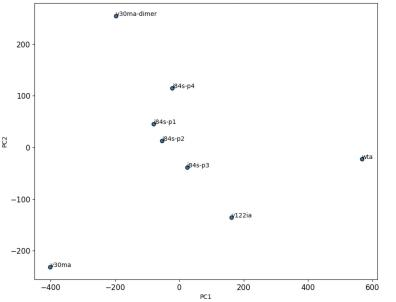}
	\caption{PCA can not distinguish I84S-p1 and I84S-p2 well.}
	\label{fig: pca}
\end{figure}

\bibliographystyle{plain}
\bibliography{Reference.bib}
\end{document}